\documentclass{amsart}
\usepackage{amsmath,amssymb,hyperref,hyperref,url}
\usepackage[vcentermath]{youngtab}
\usepackage{url,hyperref}

\newtheorem{theorem}{Theorem}[section]
\newtheorem*{theorem*}{Theorem}
\newtheorem{lemma}[theorem]{Lemma}
\newtheorem{corollary}[theorem]{Corollary}
\newtheorem{proposition}[theorem]{Proposition}

\newtheorem{claim}[theorem]{Claim}
\theoremstyle{definition}
\newtheorem{definition}[theorem]{Definition}
\newtheorem{example}[theorem]{Example}
\newtheorem{remark}[theorem]{Remark}

\def\<{{\langle}}
\def\>{{\rangle}}
\def\ZZ{{\mathbb{Z}}}

\def\W{{\mathcal{W}}}
\def\CC{{\mathbb{C}}}
\def\GL{{\operatorname{GL}}}
\def\SS{{\mathbb{S}}}
\def\S{{\mathfrak{S}}}
\def\Super{{\operatorname{Super}}}
\def\tab{{\operatorname{tab}}}

\title{Tableaux in the Whitney Module of a Matroid} \author{Andrew
  Berget}
\address{Mathematical Sciences Building\\ One Shields Ave\\
  University of California\\ Davis, CA 95616}
\email{berget@math.ucdavis.edu} \thanks{Some of this work appeared in
  the author's 2009 Ph.D. thesis from the University of
  Minnesota. Thanks are due to Victor Reiner for advising this
  thesis. Also, the author would like to thank Andrea Brini for
  sharing his work with Regonati on the letter-place approach to the
  Whitney algebra.} \date{\today}

\begin{document}

\pagestyle{plain}
\begin{abstract}
  The Whitney module of a matroid is a natural analogue of the tensor
  algebra of the exterior algebra of a vector space that takes into
  account the dependencies of the matroid.

  In this paper we indicate the role that tableaux can play in
  describing the Whitney module. We will use our results to describe a
  basis of the Whitney module of a certain class of matroids known as
  freedom (also known as Schubert, or shifted) matroids. The doubly
  multilinear submodule of the Whitney module is a representation of
  the symmetric group. We will describe a formula for the multiplicity
  hook shapes in this representation in terms of the no broken circuit
  sets.
\end{abstract}

\maketitle
\section{Introduction and Motivation}

If $V$ is a complex vector space of dimension $k$ we let $\bigwedge V$
denote the exterior algebra of $V$, $T^n(V)$ the $n$-fold tensor
product $V^{\otimes n}$ and $T(V)$ the tensor algebra $\bigotimes_{n
  \geq 0} T^n(V)$. For the moment we will only be concerned with the
\textit{$\GL(V)$-module} structure of $T(\bigwedge V)$.  We begin by
seeing how tableaux describe a basis for the $\CC$-vector space
$T(\bigwedge V)$.

Recall that the irreducible polynomial representations of $\GL(V)$ are
indexed by partitions $\lambda$ with length at most $\dim V$. We denote
the irreducible representation with highest weight $\lambda$ by
$\SS^\lambda(V)$. It follows from the Weyl character formula that the
dimension of $\SS^\lambda(V)$ is the number of column strict tableaux
of shape $\lambda$ with entries in $[\dim V]:=\{1,2,\dots,\dim
V\}$. Using Young's Rule we obtain the $\GL(V)$-module decomposition
of the tensor product of exterior products:
\[ 
{\bigwedge}^{\mu_{1}} V \otimes {\bigwedge}^{\mu_2} V \otimes \dots
\otimes {\bigwedge}^{\mu_{\ell}} V = \bigoplus_{\lambda: \ell(\lambda)
  \leq k} (\SS^\lambda V)^{\oplus K_{\lambda',\mu}} .
\]
Here $K_{\lambda',\mu}$ is the number of column strict tableaux of
shape $\lambda'$ (the conjugate partition of $\lambda$) that contain
$\mu_i$ $i$'s. From this, one easily deduces the following.
\begin{theorem}\label{thm:motivation}
  The tensor algebra $T(\bigwedge V)$ has a basis indexed by pairs of
  tableaux $(T_r,T_c)$ of the same shape where $T_r$ has strictly
  increasing rows, weakly increasing columns and entries in $[\dim V]$
  and $T_c$ is column strict with arbitrary entries.
\end{theorem}
The Whitney module of a matroid $M$, $W(M)$, will be a quotient of a
certain letter-place algebra that mimics $T(\bigwedge V)$ , but takes
into account the dependencies of the matroid $M$. Its definition is
slightly more natural than the closely related Whitney algebra of a
matroid, which was defined by Crapo, Rota and Schmitt in
\cite{crapo-schmitt}. In the final section of \cite{crapo-schmitt} the
Whitney module of a matroid is roughly described in passing. The goal
of this paper is to begin to investigate how tableaux play a role in
describing the structure of $W(M)$. Our main result is that the
obvious spanning set of $W(M)$ is a basis when $M$ is a
\textit{freedom matroid} (also known as Schubert, or shifted
matroids). We will elaborate on and prove the following result.

\begin{theorem}
  Let $M$ be a freedom matroid on $n$ elements. There is a basis for
  its Whitney module indexed by pairs of tableaux $(T_r,T_c)$ where
  $T_r$ and $T_c$ have the same shape and
  \begin{enumerate}
  \item $T_r$ is row strict with entries in $[n]$,
  \item every row of $T_r$ indexes an independent set of $M$, and
  \item $T_c$ is column strict.
  \end{enumerate}
\end{theorem}

We will also precisely state and prove the following result.
\begin{theorem}
  In the complexified doubly multilinear submodule of the Whitney
  module of $M$, a basis for the hook shaped isotypic components are
  determined by the no broken circuit complex of $M$.
\end{theorem}

The paper is organized as follows. First we define the super algebra
$\Super(L^-|P^+)$ and recall the standard basis theorem of
Groshhans--Rota--Stein \cite{grs}.  We then define the Whitney algebra
and Whitney module of a matroid and see that the latter is spanned by
certain elements indexed by pairs of tableaux, as in the standard basis
theorem, with the additional condition that the rows of the row strict
tableau index independent subsets of the matroid. We will then define
freedom matroids and prove that the given spanning set for its Whitney
module forms a basis. Following that, we will define the doubly
multilinear submodule of the Whitney module and describe an action of
the symmetric group on it. After complexifing this submodule, we will
give a formula for the multiplicity of irreducible symmetric group
modules indexed by hook shapes. The formula will be in terms of the
certain no broken circuit subsets of $M$.

\section{Letter-Place Algebras}
In this section we show how to view the tensor algebra of the exterior
algebra of a finite set as a letter-place superalgebra. All of the
definitions come from \cite{grs} and we use the main result there, the
standard basis theorem, to describe a combinatorial basis of this
object in terms of pairs of tableaux. This gives a concrete
explanation of our motivating result Theorem~\ref{thm:motivation}.

\subsection{Exterior Algebra}
Let $E$ be a finite set and $\bigwedge E$ be the exterior algebra of
the free $\ZZ$-module with basis $E$. We will write decomposable
elements as $w = e_{i_1} e_{i_2} \dots e_{i_k}$, to avoid notational
clutter. This is a graded commutative algebra, which means that
$\bigwedge E$ is the direct of the $i$-fold exterior products
\[
\bigwedge E = \bigoplus_{0 \leq i } {\bigwedge}^i E
\]
and if $w \in \bigwedge^i E$ and $w' \in \bigwedge^j E$ then
\[
ww' = (-1)^{ij} w'w.
\]
If $w \in \bigwedge E$ is homogeneous then we denote the degree of the
piece which $w$ is in by $|w|$.  The tensor algebra of the graded
algebra $\bigwedge E$ is the direct sum of the tensor products
$T^1(\bigwedge E)$, $T^2(\bigwedge E)$, \dots.  Each of the summands
has its own product, called the internal product, which is induced by
the rule
\[
(w \otimes w') \times (u \otimes u') = (-1)^{|w'||u|}(wu \otimes w'u').
\]

The exterior algebra of $E$ is a graded commutative Hopf algebra, with
coproduct
\[
\delta:
\bigwedge E \to \bigwedge E \otimes \bigwedge E
\]
induced by the rule $\delta(e) = 1\otimes e + e \otimes 1$. We will
not need the definitions of the counit or antipode here.
Given an element $w \in \bigwedge E$ we write its coproduct using
Sweedler notation
\[
\delta(w) = \sum_{w} w_{(1)} \otimes w_{(2)}.
\]
The iterated coproduct $\delta^{(n)}: \bigwedge E \to T^n(\bigwedge
E)$ is defined by the conditions that $\delta^{(1)}$ equals the
identity map, $\delta^{(2)} = \delta$ and $\delta^{(n)} = (\delta
\otimes 1) \circ \delta^{(n-1)}$. The iterated coproduct
$\delta^{(n)}$ is the sum of its homogeneous pieces
\[
\delta^{(\alpha)} : \bigwedge E \to ({\bigwedge}^{\alpha_1} E) \otimes
({\bigwedge}^{\alpha_2} E) \otimes \dots \otimes
({\bigwedge}^{\alpha_n} E)
\]
where $\alpha = (\alpha_1,\alpha_2,\dots,\alpha_n)$ is a composition
with $n$ parts. The image $w$ under $\delta^{(\alpha)}$ is the
$\alpha$-th coproduct slice of $w$.


\subsection{Letter-Place Algebras}
The goal of this section is to view $T(\bigwedge E)$ as one of the
letter-place algebras of Grosshans--Rota--Stein \cite{grs}. This will
be done by constructing two algebras, one of non-commutative letters
the other of commutative places, and from these defining a new algebra
of letter-place pairs.

We will not review the complete definition of the letter-place
algebras, since this some amount of work to do precisely. Instead, we
will take those pieces of the definitions suited for our needs,
hinting at the form of the general definitions.

In this section we will declare the elements of $E$ to be negatively
signed and refer to them as \textit{negative letters}. To emphasize
that we are viewing $E$ in this way, we will denote it by $L^-$ (or
sometimes $L^-_E$ when we must emphasize the set $E$).  Let $P^+ =
\{p_1,p_2,p_3,\dots\}$ denote the infinite set of \textit{positively
  signed places}. We associate to $L^-$ and $P^+$ two algebras, the
exterior algebra and the divided power algebra, respectively, and
associate to these the so-called letter-place algebra. Will write the
exterior algebra of $L^-$ as $\Super(L^-)$, which is constructed
exactly as in the previous section.

We recall the definition of the divided power algebra. Associate
to each element $p \in P^+$ an infinite sequence of \textit{divided
  powers} $$p^{(0)},p^{(1)}, p^{(2)},\dots.$$ We let $\ZZ\<P^+_d\>$
denote the free algebra generated by the divided powers.  Let
$\Super(P^+)$ be the quotient of $\ZZ\<P^+_d\>$ by the two-sided ideal
generated by the elements of the form
\[
p^{(j)} p^{(k)} - \binom{j+k}{k} p^{(j+k)}, \quad p^{(j)}q^{(k)} -
q^{(k)}p^{(j)},\quad p^{(0)} - 1,
\]
for any $p,q \in P^+$.  The elements $p^{(j)}$ are meant to behave
like $p^i/i!$ in a symmetric algebra. We can endow $\Super(P^+)$ with
a coalgebra structure by defining the coproduct of $p^{(j)}$, $p \in
P$ as
\[
\delta(p^{(j)}) = 1 \otimes p^{(j)} + p^{(1)} \otimes p^{(j-1)} +
\dots + p^{(j-1)} \otimes p^{(1)} + p^{(j)} \otimes 1.
\]
As before we write the coproduct of an arbitrary element of $q \in
\Super(P^+)$ in Sweedler notation,
\[
\delta(q) = \sum_q q_{(1)} \otimes q_{(2)}.
\]
It is clear that $\Super(P^+)$ is graded and commutative in the usual
sense of commutative algebra.

Finally we are in a positive to define the letter-place algebra
$\Super(L^-|P^+)$. Let $(L^-|P^+)$ denote the set of letter-place
pairs:
\[ (L^-|P^+) = \{ (e|p) : e \in L^-, p \in P^+\}.
\]
Since $L^-$ consists of negatively signed variables and $P^+$ consists
of positively signed variables, we declare the letter-place pairs to
be negatively signed. We define $\Super(L^-|P^+)$ to be the exterior
algebra of the set $(L^-|P^+)$. 

To describe a standard basis of $\Super(L^-|P^+)$ we define a certain
bilinear map
\[
\Omega:\Super(L^-) \times \Super(P^+) \to \Super(L^-|P^+),
\]
called the \textit{Laplace pairing}, according to a sequence of rules.
\begin{itemize}
\item[R1.] $\Omega(1,1) = 1$.
\item[R2.] If $e \in L^-$ and $p \in P^+$ then $\Omega(e,p^{(1)}) =
  (e|p)$.
\item[R3.] If $w$ and $p$ are not in the same graded piece of
  $\Super(L^-)$ and $\Super(P^+)$, respectively, then $\Omega(w,p) =
  0$.
\item[R4.] If $\delta(p) = \sum_{(p)} p_{(1)} \otimes p_{(2)}$ then
  $$\Omega(ww',p) = \sum_{(p)} \Omega(w,p_{(1)})\Omega(w',p_{(2)}).$$
\item[R$4'$.] If $\delta(w) = \sum_{(w)} w_{(1)} \otimes w_{(2)}$ then
  $$\Omega(w,pp') = \sum_{(w)} \Omega(w_{(1)},p)\Omega(w_{(2)},p').$$
\end{itemize}
That the rules R4 and R$4'$ are equivalent follows from a series of
technical checks, which was done \cite{grs}. There they give a more
general definition of the Laplace pairing was given that includes the
possibility of both positively and negatively signed letters and
places. In the future we will denote the Laplace pairing of $w$ and
$q$ by $(w|q)$, so that elements of the form $(efg|p_1^{(2)}p_2)$ make
sense.

\begin{proposition}\label{prop:super-tensor}
  Viewing $T(\bigwedge E)$ as a $\ZZ$-algebra with the internal
  product, there is a surjection of $\ZZ$-algebras
  \[
  T(\bigwedge E) \to \Super(L^-,P^+),
  \]
  that maps a tensor $w_1 \otimes w_2 \otimes \dots \otimes w_n$ to
  the product
  \[
  (w_1|p_1^{|w_1|})(w_2|p_2^{|w_2|}) \dots (w_n|p_1^{|w_n|}).
  \]
\end{proposition}

\begin{example}
  Suppose that $e,f \in E = L^-$. Then
  \[
  e \otimes e \otimes f \mapsto (e|p_1)(e|p_2)(f|p_3).
  \]
  We can verify that $ ef + fe \mapsto 0$. We know that
  $$\delta(p_1^{(2)}) = 1 \otimes p_1^{(2)} + p_1^{(1)} \otimes
  p_1^{(1)} + p_1^{(2)} \otimes 1,$$ hence
  \[
  ef \mapsto (e|1)(f|p_1^{(2)}) + (e|p_1)(f|p_1) + (e|p_1^{(2)})(f|1)
  = (e|p_1)(f|p_1) = -(f|p_1)(e|p_1),
  \]
  since $\Super(L^-|P^+)$ is an exterior algebra. 

  The preimage of $(e|p_2)$ consists of elements of the form 
  $$
  1  \otimes e \otimes 1 \otimes \dots \otimes 1
  $$
  since, e.g.,
  \[
  1 \otimes e \mapsto (1|p_1^{(0)})(e|p_2) = (1|p_1^{(0)})(e|p_2) =
  (e|p_2).
  \]
\end{example}

Given a composition $\alpha = (\alpha_1,\dots,\alpha_n)$ we define
\[
p^{(\alpha)} := p_1^{(\alpha_1)} p_2^{(\alpha_2)}\dots p_n^{(\alpha_n)}
\]
\begin{proposition}\label{prop:coproduct-slice}
  Let $w \in \bigwedge E$ be a decomposable element. The image of the
  coproduct slice $\delta^{(\alpha)} (w)$ in $\Super(L^-|P^+)$ is the
  Laplace pairing $(w|p^{(\alpha)})$.
\end{proposition}
\begin{proof}
  The result is easy to verify if $w = e \in L^-$. If $|w|>1$ then we
  may write $w = w'w''$. By the homogeneity of the coproduct and
  induction we have
  \[
  \delta(w'w'') = \sum_{\beta+\gamma = \alpha}\delta^{(\beta)}(w')
  \delta^{(\gamma)}(w'') \mapsto \sum_{\beta+\gamma = \alpha}
  (w'|p^{(\beta)})(w''| p^{(\gamma)}).
  \]
  Since we have $\delta( p^{(\alpha)}) = \sum_{\beta+ \gamma = \alpha}
  p^{(\beta)} \otimes p^{(\gamma)}$, we can rule R4 to write this as
  $(w'w''| p^{(\alpha)})$.
\end{proof}

\subsection{The Standard Basis Theorem}\label{sec:std basis}
From the computation of the $\GL(V)$-module structure of $T(\bigwedge
V)$ in the introduction, we expect $\Super(L^-|P^+)$ to have a basis
indexed by pairs of tableaux of the same shape where one is row strict
and the other is column strict. This is the case, and in this section
we recall how to construct this basis.

Let $\lambda$ be a decreasing sequence of non-negative integers;
\textit{a partition}. The \textit{length} of $\lambda$ is the number
of positive integers in the sequence. We will identify $\lambda$ with
its \textit{Young frame}, which is a collection of boxes, north-east
justified, the number of boxes in the $i$-th row being equal to
$\lambda_i$. Denote the total number of boxes in the Young frame of
$\lambda$ by $|\lambda|$. A \textit{tableau} $T$ is a filling of the
elements of $A$ into the boxes of a partition $\lambda$. If $T$ is a
tableau we will call the partition $\lambda$ the \textit{shape} of $T$
and write $sh(T)$. For example
\[
\young(32542,3332,12,12)
\]
is a tableau on $\{1,2,3,4,5\}$ whose shape is $(5,4,2,2)$. The
\textit{content} of a tableau is the number of $1$'s, the number of
$2$'s, \dots that appear in the filling. We will write the content of
a tableau as a composition whose $i$-th part is the number of $i$'s in
the filling of the tableau. Thus the tableau above content
$(2,5,4,1,1,0,0,\dots)$. A \textit{column strict tableau} is a tableau
where the numbers in each row weakly increase and the numbers in each
column strictly increase. A \textit{row strict tableau} is a tableau
where the numbers in each column weakly increase and the numbers in
each row strictly increase. We will call a tableaux $T$ a standard
Young tableaux if it is both row and column strict and has entries in
$|sh(T)|$.

Let $T$ and $S$ be tableau of the same shape $\lambda$ and length
$\ell$. Let the numbers in the $i$-th row of $T$ be
$t_1,\dots,t_{\lambda_i}$, in order. Define $w_i$ to be the product in
$\Super(L^-)$ of the elements indexed by $t_1, \dots, t_{\lambda_i}$,
i.e., $w_i = e_{t_1} \dots e_{t_{\lambda_i}}$. Let
$s_1,\dots,s_{\lambda_i}$ be the elements in the $i$-th row of $S$, in
order. Define $q_i$ to be the product in $\Super(P^+)$ of the elements
indexed by $s_1,\dots,s_{\lambda_i}$, where if $s_{j} = \dots =
s_{j+k-1}$ is a maximal string of equal entries then we take
$p_{s_j}^{(k)}$ instead of the product $p_{s_j} \dots
p_{s_{j+k-1}}$. For example, if
\[
T = \young(1234,123,3) \quad S = \young(3343,555,6)
\]
then
\[
w_1 =w_2= e_1e_2e_3,\quad w_2 = e_3, \qquad\qquad q_1 =
p_3^{(2)}p_4^{(1)}p_3,\quad q_2 = p_5^{(3)}, \quad q_3 = p_6^{(1)}.
\]
We $\tab(T|S)$ by the formula
\[
\tab(T|S) = (w_1|q_1)(w_2|q_2) \dots (w_\ell|q_\ell) \in
\Super(L^-|P^+),
\]
which makes sense according to our definition of the Laplace
pairing. We call such an element a tableaux in $\Super(L^-|P^+)$.  We
are finally in a position to state the main result of this section.
\begin{theorem}[Grosshans--Rota--Stein {\cite{grs}}]\label{thm:std basis}
  The elements $\tab(T_r|T_c) \in \Super(L^-|P^+)$ where 
  \begin{enumerate}
  \item $T_r$ and $T_c$ are tableaux of the same shape,
  \item $T_r$ is row strict with entries in $[n]$,
  \item $T_c$ is column strict,
  \end{enumerate}
  form a basis for the free module $\Super(L^-|P^+)$. We will call
  such tableaux \textit{standard}.

  In the expansion of $\tab(T|S)$ as a sum of standard tableaux
  $\sum_i c_i\tab(T_i|S_i)$ we have that the shape of each $T_i$ is
  larger than or equal to the shape of $T$, in dominance
  order. Further, the content of every $T_i$ is equal to the content
  of $T$ and the content of every $S_i$ is equal to the content of
  $S$.
\end{theorem}

\section{The Whitney Algebra and Module of a Matroid}
We assume that the reader is familiar with the basic concepts in
matroid theory (see, e.g., \cite{white}).

In this section we define the Whitney algebra and Whitey module of
matroid. We then show that if $M$ is realizable over $\CC$ then the
standard tableau pairs of the previous section have nonzero image in
the Whitney module of $M$ if and only if the rows of the first tableau
index independent sets of $M$.
\subsection{Definitions}
Let $M$ be a matroid on $E$ of rank $r(M)$. Decomposable elements of
$\bigwedge E$ are given by words on $E$. We say that a decomposable
element $e_{i_1}e_{i_2} \dots e_{i_k} \in \bigwedge E$ is a dependent
word if $\{i_1,i_2,\dots,i_k\}$ is a dependent set in $M$. Likewise we
define independent words.
\begin{definition}[Crapo--Schmitt {\cite{crapo-schmitt}}]
  The Whitney algebra of a matroid $M$, denoted $\W(M)$, is the
  quotient by $T(\bigwedge E)$ by the ideal generated by the elements 
  \[
  \delta_\alpha(w)
  \]
  where $w$ is a dependent word in $M$ and $\alpha$ is a composition
  of $|w|$.
\end{definition}

The following definition was also given by Brini and Regonati
(unpublished, \cite{brini}).
\begin{definition}
  The Whitney module of a matroid $M$, denoted $W(M)$, is defined to
  be the quotient of $\Super(L^-|P^+)$ by the two-sided ideal
  generated by the elements
  \[
  (w|p^{(\alpha)})
  \]
  where $w$ is a dependent word in $M$ and $\alpha$ is a composition of
  $|w|$.
\end{definition}
Using rule R4 it is clear that it is sufficient to take $w$ to be the
word of a circuit of $M$ in the definition of $W(M)$.

It is obvious from Proposition~\ref{prop:coproduct-slice} that there
is a surjective map $\mathcal{W}(M) \to W(M)$ that takes the internal
product of $\mathcal{W}(M)$ to the product that $W(M)$ inherits as a
quotient of an exterior algebra. One can think of $W(M)$ as being
obtained from $\W(M)$ by appending a half-infinite string of the form
$1 \otimes 1 \otimes 1 \otimes \dots$ to the right of every element of
$\W(M)$ (see the comments at the end of \cite{crapo-schmitt}).

Since $\Super(L^-|P^+)$ is a graded commutative algebra and the ideal
defining $W(M)$ is homogeneous, $W(M)$ inherits a grading. Each of the
graded pieces is a finitely generated $\ZZ$-module, and hence can be
written as the direct sum of a free part and a torsion part.
\begin{proposition}
  There is a direct sum decomposition 
  \[
  W(M) = W(M)_{free} \oplus W(M)_{tor}
  \]
  where $W(M)_{free}$ is free and $W(M)_{tor}$ is torsion.
\end{proposition}
It is a basic example of Crapo and Schmitt \cite{crapo-schmitt} that
if $M$ is not realizable over a field of characteristic zero then
$W(M)_{tor}$ can be non-zero. It is unknown if $W(M)_{tor}$ is zero
when $M$ is realizable over a field of characteristic zero.

\subsection{Tableaux in the Whitney Module}
Let $T$ be a tableau with entries in $[n]$, and $S$ an arbitrary
tableau of the same shape.  Since $W(M)$ is a quotient of
$\Super(L^-|P^+)$ we can project the elements $\tab(T|S)$ of
Section~\ref{sec:std basis} into $W(M)$.  Abusing notation, we denote
the image of $\tab(T|S)$ in $W(M)$ by $\tab(T|S)$. 

Note that every standard tableaux $\tab(T|S)$ can be written as
\[
(w_1| p^{(\alpha^1)})(w_2| p^{(\alpha^2)}) \dots (w_\ell|
p^{(\alpha^\ell)}),
\]
where $\alpha^i$ is a composition of $|w_i|$ and $|w_1| \geq |w_2|
\geq \dots \geq |w_\ell|$.
\begin{proposition}
  The image of an arbitrary tableaux $\tab(T|S)$ in $W(M)$ is zero if
  some row of $T$ indexes a dependent set of $M$.
\end{proposition}
\begin{proof}
  This follows since each tableau is a product of elements of the form
  $(w| p^{(\alpha)})$, and we know that this element is zero in the
  Whitney module if $w$ is a dependent word.
\end{proof}

The main theorem of this section is the following result.
\begin{theorem}\label{thm:gamas}
  Let $S$ be a column strict tableau. If $M$ is realizable over $\CC$,
  the image of image of the tableaux $\tab(T|S)$ in $W(M)$ is non-zero
  if and only if the rows of $T$ index independent sets of $M$.
\end{theorem}

Before we proceed with the proof we set up a nice corollary, that
gives us a simple check of whether such a tableaux exists, having
prescribed the content and shape of $T$. The content of the $T$
determines a parallel extension of the labeled matroid $M$. Indeed if
the content of $T$ is $\mu$ (a composition with $n$ parts) then the
parallel extension is $M_\mu$, which has the $\mu_i$ copies of the
element $e_i$.

The \textit{rank partition} of a matroid $M$ is the sequence of
numbers $\rho(M) = (\rho_1,\rho_2,\dots)$ determined by the condition
that
\[
\rho_1 + \rho_2 + \dots + \rho_k
\]
is the size of the largest union of $k$ independent subsets of
$M$. This definition was first given by Dias da Silva in \cite{dds}
where he proved the following result.
\begin{theorem}[Dias da Silva {\cite{dds}}]\label{thm:dds}
  The rank partition of matroid is a partition. There is a partition
  of the ground set of a matroid $M$ into independent sets of size
  $\lambda_1 \geq \lambda_2 \geq \dots $ if and only if $\lambda \leq
  \rho(M)$ in dominance order.
\end{theorem}
The following corollary is now immediate from the theorem.
\begin{corollary}
  There is a non-zero tableaux $\tab(T|S) \in W(M)$ of shape
  $\lambda$, where $S$ is columns strict and $T$ has content $\mu$, if
  and only if $\lambda \leq \rho(M_\mu)$ in dominance order.
\end{corollary}

To prove Theorem~\ref{thm:gamas} we need a lemma.
\begin{lemma}\label{lem:gamas}
  Suppose that $S$ is a fixed column strict tableaux of shape
  $\lambda$ whose first row gives rise to the element $p^{(\alpha)}
  \in \Super(P^+)$. Let $S'$ denote $S$ with its first row removed.

  Define two vector spaces: $X$ is the subspace of $\Super(L^-|P^+)
  \otimes \CC$ spanned by standard tableaux $\tab(T|S)$ where $T$ has
  first row equal containing the numbers
  $\{1,2,\dots,\lambda_1\}$. The second vector space $X'$ is the
  subspace of $\Super(L^-|P^+) \otimes \CC$ spanned by any standard
  tableaux $\tab(T'|S')$.

  Then multiplication by
  \[
  (e_{ 1}e_2 \dots e_{\lambda_1} | p^{(\alpha)} )
  \]
  induces an isomorphism of vector spaces $X' \to X$.
\end{lemma}
\begin{proof}
  This follows directly from the standard basis theorem, since this
  map takes bases to bases.
\end{proof}
\begin{proof}[Proof of Theorem~\ref{thm:gamas}]
  Let $f: E \to V$ is a realization of $M$, where $V$ is a complex
  $n$-dimensional vector space. After choosing a basis for $V$, we can
  identify $\bigwedge V$ with $\Super(L^-)$ and since $f$ is a
  realization of $M$, the mapping $(e|p) \mapsto
  (f(e)|p)$ gives rise to a map of algebras
  \[
  f:W(M) \to \Super(L^-|P^+) \otimes \CC, \qquad (w|p^{(\alpha)})
  \mapsto (f(w)|p^{(\alpha)}),
  \]
  (compare Proposition 6.3 in \cite{crapo-schmitt}). This map will almost
  always fail to be surjective since $M$ will typically not have rank
  $n$. Note that $\Super(L^-|P^+)$ comes with a left $\GL_n(\CC)$
  action, induced by the natural action of $\GL_n(\CC)$ on
  $\Super(L^-) \otimes \CC$. Taking a limit, there is a corresponding
  action of $n \times n$ complex matrices on $\Super(L^-|P^+)$.

  Suppose that we have the tableaux
  \[
  \tab(T|S) = (w_1| p^{(\alpha^1)})(w_2| p^{(\alpha^2)})\dots (w_\ell|
  p^{(\alpha^\ell)}) \in W(M)
  \]
  where $S$ is a column strict tableaux of shape $\lambda$, length
  $\ell$ and $|w_i| = \lambda_i$. Applying the map $f$ from above we
  obtain
  \begin{multline*}
  f(\tab(T|S)) = (f(w_1)| p^{(\alpha^1)})(f(w_2)| p^{(\alpha^2)})\dots
  (f(w_\ell)| p^{(\alpha^\ell)})\\ \in \Super(L^-|P^+) \otimes \CC.
  \end{multline*}
  We will prove by induction on the length of $\lambda$ that
  $f(\tab(T|S))$ is not zero provided that $f(w_i) \neq 0$. Since the
  image of $\tab(T|S)$ is not zero, it must be that $\tab(T|S)\neq 0$
  in $W(M)$.

  Let $A$ be a generic matrix such that
  \[
  Af(w_1) = e_{1} e_2  \dots e_{\lambda_1}
  \]
  Since $A$ is generic, each of element $A(f(w_i))$ is not zero and
  decomposable in $\Super(L^-)\otimes \CC$. By Lemma~\ref{lem:gamas},
  we have that
  \[
  (e_{1}e_2 \dots e_{\lambda_1}| p^{(\alpha^1)}) (Af(w_2)|
  p^{(\alpha^2)})\dots (A(f(w_\ell))| p^{(\alpha^\ell)})
  \]
  is not zero if and only if 
  \[
  (Af(w_2)| p^{(\alpha^2)})\dots (A(f(w_\ell))| p^{(\alpha^\ell)})
  \]
  is not zero. This is not zero by induction. It only remains to check
  the basis step. This follows since
  \[
  (e_{i_1} \dots e_{i_k} |p^{(\alpha)}),
  \]
  $i_1<\dots<i_k$, is a basis element of $\Super(L^-|P^+) \otimes
  \CC$, according to the standard basis theorem.
\end{proof}
\begin{remark}
  One cannot remove the hypothesis that $S$ is column strict.  For
  example if $M$ is a boolean matroid (i.e., $W(M) = \Super(L^-|P^+)$)
  and
  \[
  T = S = \young(123,123)
  \]
  then
  \[
  \tab(T|S) = (e_1e_2e_3| p_1^{(1)}p_2^{(1)}p_3^{(1)})^2 = 0
  \]
  even though $(e_1e_2e_3| p_1^{(1)}p_2^{(1)}p_3^{(1)}) \neq 0$ in
  $W(M)$.
\end{remark}
It is unknown if Theorem~\ref{thm:gamas} holds for any realizable
matroid.
\section{Freedom Matroids}
In this section we define \textit{freedom matroids} and show that the
obvious spanning set for their Whitney modules are bases.
\subsection{Definition of a Freedom Matroid}
\begin{definition}
  We will denote the \textit{direct sum of a matroid $M$ with the one
    element rank one matroid} on the set $\{e\}$ by $M \oplus e$.

  Let $M$ be a matroid of rank larger than $0$. The \textit{truncation
    of $M$ to rank $k \leq r(M)$} is the matroid whose bases are those
  independent sets of $M$ with size $k$. The truncation of $M$ to rank
  $r(M)-1$ will be denoted $T(M)$.

  The \textit{principal extension of a matroid $M$ along the improper
    flat} is the matroid $M+e$ obtained by truncating the direct sum
  $M \oplus e$ to the the rank of $M$. That is $M+e = T(M \oplus e)$.
\end{definition}
We think of $M+e$ as adding a new element generically to $M$ without
increasing its rank.
\begin{definition}
  For $i \in \{0,1\}$ define $M_{(i)}$ to be the rank $i$ matroid on
  one element $e_1$.  Let $s$ be a binary sequence of length $n> 1$
  and $s'$ be the sequence obtained by deleting its last entry
  $s_n$. Define a matroid $M_s$ on the set $\{e_1, e_2,\dots,e_n\}$ by
  setting
  \[
  M_s = \begin{cases}
    M_{s'} \oplus e_n & s_n = 1, \\
    M_{s'} + e_n & s_n = 0.
    \end{cases}
  \]
  A \textit{freedom matroid} is a labeled matroid of the form $M_s$
  for some binary sequence $s$.
\end{definition}
\begin{example}
  The freedom matroid associated with the sequence $(1,0,1,0,1,0)$ is
  \[
  \left(\left(\left(\left(M_1 + e_2\right) \oplus e_3\right) +
      e_4\right) \oplus e_5\right) + e_6
  \]
  where $M_1$ is the one element rank one matroid with ground set
  $\{e_1\}$. It is represented linearly by the columns of the matrix
  \[\begin{bmatrix}
    1 & 1 & * & * & * & * \\
      &   & 1 & 1 & * & * \\
      &   &   &   & 1 & 1
    \end{bmatrix}\] where the blank entries are zero and the $*$'d
    entries are generic elements of $\CC$.
\end{example}

Freedom matroids arise in many contexts. They are the matroids
associated to the generic point of a Schubert strata of a
Grassmannian. They are known to be the matroids whose independence
complexes are shifted. They are special cases of lattice-path
matroids.
\subsection{The Whitney Module of a Freedom Matroid}
We can now state and prove our first main theorem.
\begin{theorem}\label{thm:whitney-freedom}
  Let $M$ be a freedom matroid (in particular, the ground set of $M$
  is ordered). The Whitney module $W(M)$ is free and a basis consists
  of those standard tableaux $\tab(T_r|T_c)$ where the rows of $T_r$
  index independent sets of $M$.
\end{theorem}
We already know that these tableaux span $W(M)$, so it suffices to
prove that $W(M)$ is free and the stated elements are linearly
independent when $M$ is a freedom matroid.

This result does not hold in general, as we will see in
Section~\ref{sec:universal}. Since uniform matroids are a special case
of freedom matroids (they are associated to sequences of the form
$(1,1,\dots,1,0,0,\dots,0)$) we see that the above theorem describes
the Whitney module of uniform matroids. This result was obtained in
2000 by Crapo and Schmitt \cite{crapo}.

\begin{lemma}\label{lem:direct-sum}
  Let $M$ and $N$ be matroids. There is an isomorphism of graded
  algebras
  \[
  W(M \oplus N) \approx W(M) \otimes W(N),
  \]
  where the $\otimes$ is the super tensor product of graded algebras.
\end{lemma}
Recall that the super tensor product of two graded $\ZZ$-algebras $H$
and $L$ has module structure given by the $\ZZ$-module $H \otimes L$
and algebra structure induced by the formula for multiplying
homogeneous elements $$(h \otimes l)(h' \otimes l') = (-1)^{|h'||l|}
hh' \otimes ll'.$$
\begin{proof}
  We use the following fact: If $H$ and $L$ are graded algebras and
  $I$ and $J$ are homogeneous ideals of $H$ and $L$, respectively,
  then as graded algebras
  \begin{equation}
    \label{eq:tensor product of algebras}
    H/I \otimes L/J \approx \frac{H \otimes L}{H \otimes J + I \otimes
      L}.    
  \end{equation}

  In the present situation, we suppose that $M$ is a matroid on $E$
  and $N$ is a matroid on $F$. Let $L_E^-$ denote the set of
  negatively signed letters $E$ and $L_F^-$ denote the set of
  negatively signed letters $F$. It is well known that
  \begin{equation}
    \label{eq:tensor product of super algebras}
    \Super(L^-_E|P^+) \otimes \Super(L^-_F |P^+) \approx \Super(L_E^-
    \cup L_F^- | P^+).    
  \end{equation}
  If $I$ is an ideal of $\Super(L^-_E|P^+)$ then under this
  isomorphism we have
  \begin{equation}\label{eq:ideal}
    I \otimes \Super(L^-_F |P^+) \to \Super(L^-_E \cup L^-_F |P^+) \cdot
    I,    
  \end{equation}
  i.e., the ideal $I \otimes \Super(L^-_F |P^+)$ maps to the ideal
  generated by $I$ in $\Super(L^-_E \cup L^-_F |P^+)$.

  To complete the proof all we need to note is that a circuit of $M
  \oplus N$ is either a circuit of $M$ or a circuit of $N$. Thus,
  \begin{multline}\label{eq:circuits of sum}
    \< (w|p^{\alpha}) : w\textup{ is a circuit of }M \oplus N \>= \<
    (w|p^{\alpha}) : w\textup{ is a circuit of }M \>\\+ \<
    (w|p^{\alpha}) : w\textup{ is a circuit of }N \>
  \end{multline}
  where here $\<-\>$ denotes taking the ideal in $\Super(L^-_E \cup
  L^-_F | P^+)$ generated by the elements $-$.

  Combining equations \eqref{eq:tensor product of
    algebras}---\eqref{eq:circuits of sum} with the definition of the
  Whitney module we have $W(M) \otimes W(N) \approx W(M \oplus N)$.
\end{proof}
\begin{corollary}
  If $W(M)$ and $W(N)$ are free $\ZZ$-modules, then so is $W(M \oplus
  N)$.
\end{corollary}
\begin{lemma}\label{lem:truncation}
  Suppose that $M$ is a matroid of rank larger than $0$ and $W(M)$ is
  free, so that $W(M)$ has a basis $\mathcal{B}$ consisting of some
  standard tableaux. Then $W(T(M))$ is free and has a consisting of
  those standard tableaux in $\mathcal{B}$ whose first row has length
  less than $r(M)$.
\end{lemma}
\begin{proof}
  It is easy to convince oneself that
  \[
  W(T(M)) = W(M)/\< \tab(T|S): \operatorname{sh}(T)=\lambda, \lambda_1 =
  r(M)\>
  \]
  where $\<-\>$ denotes taking the two sided ideal in $W(M)$ generated
  by the elements $-$. Now, by the standard basis theorem we know that
  for non-standard tableaux $\tab(T|S)$,
  \[
  \tab(T|S) = \sum_{i=1}^m c_i \tab(T_i|S_i)
  \]
  where $\operatorname{sh}(T_i) \geq \operatorname{sh}(T)$ in
  dominance order and each $(T_i|S_i) \in \mathcal{B}$. This implies
  that if the first row of $T$ has length $r(M)$ then the first row of
  every $T_i$ has length at least $r(M)$. It follows that
  \[
  W(T(M)) = W(M)/\< \tab(T|S): \operatorname{sh}(T)=\lambda, \lambda_1 =
  r(M), \tab(T|S) \in \mathcal{B}\>
  \]
  Since $W(M)$ is free, the claim now follows.
\end{proof}
\begin{proof}[Proof of Theorem~\ref{thm:whitney-freedom}]
  The theorem will follow by induction. We will prove that if the
  result holds for a matroid $M$ on $\{e_1 < e_2 < \dots < e_n\}$ then
  it also holds for $M \oplus e_{n+1}$ and $T(M)$. Since freedom
  matroids are closed under these operations, it is sufficient to
  prove the result for the two one element matroids, which is trivial.

  For a positive integer $m$ let $W(M)_{\leq m}$ be the subalgebra of
  $W(M)$ generated by letter place pairs $(e|p)$ where $p \in
  \{p_1,p_2,\dots,p_m\}$. By our assumption on $W(M)$ we see that a
  basis for $W(M)_{\leq m}$ consists of those tableaux $\tab(T_r|T_c)$
  where $T_r$ is row strict, $T_c$ is column strict, every row of
  $T_r$ indexes an independent set of $M$ and every entry of $T_c$ is
  at most $m$. This is a module of finite rank, since the shape of
  every tableau appearing must fit into a $m$-by-$r(M)$ box.

  In light of Lemma~\ref{lem:direct-sum}, it is straightforward to
  convince oneself that 
  \[
  W(M \oplus e_{n+1})_{\leq m} = W(M)_{\leq m} \otimes
  W(\{e_{n+1}\})_{\leq m},
  \] 
  where $\{e_{n+1}\}$ denotes the rank one element one element
  matroid. It follows that the the rank of $W(M)_{\leq m} \otimes
  W(\{e_{n+1}\})_{\leq m}$ as a $\ZZ$-module is the number of pairs
  \[
  \left((T_r,T_c) \quad , \quad (e_{n+1}|p_{i_1})(e_{n+1}|p_{i_2})
    \dots (e_{n+1}|p_{i_k}) \right)
  \]
  where $T_r$ is row strict, $T_c$ is column strict, every row of
  $T_r$ indexes an independent set of $M$, the entries of $T_c$ are at
  most $m$, and $1 \leq i_1 < i_2 < \dots < i_k \leq m$. From this
  pair we produce two new tableaux $T_r'$ and $T_c'$ where
  \[
  T_c' = (((T_c \leftarrow i_1) \leftarrow i_2) \leftarrow \dots \leftarrow
  i_k)
  \]
  is obtained by the usual Robinson--Schensted row insertion
  and $T_r'$ is obtained from $T_r$ by recording the new boxes of
  $T_c'$ with $n+1$'s.

  It follows from the Super RSK correspondence \cite{senato1,senato2}
  that this map is bijective and its image consists of pairs of
  tableaux $(T,S)$ of the same shape where $T$ is row strict with
  entries in $[n+1]$, $S$ is column strict with entries in $[m]$ and
  the rows of $T$ index independent subsets of $M \oplus
  e_{n+1}$. This is because $e_{n+1}$ is in no circuit of the direct
  sum. This complete the proof of the induction for direct sums.

  Suppose now that $W(M)$ is free and a basis consists of those
  standard tableaux $(T_r|T_c)$ such that every row of $T_r$ indexes
  an independent set of $M$. Then Lemma~\ref{lem:truncation} proves
  that the same statement holds for $W(T(M))$.
\end{proof}

\section{The Doubly Multilinear Submodule of $W(M)$}\label{sec:universal}
The doubly multilinear submodule of $W(M)$ is the submodule generated
by elements of the form
\[
(e_1 | p_{\sigma(1)} )(e_{2} | p_{\sigma(2)} ) \dots (e_{n} | p_{\sigma(n)} )
\in W(M)
\]
where $\sigma$ is any permutation in the symmetric group $\S_n$. We
denote this submodule by $U(M)$. There is a right action of $\S_n$ on
$U(M)$, by permuting places. We will primarily be interested in the
complexified version $\CC \otimes U(M)$, where $M$ is a matroid
realizable over the complex numbers.

The module $U(M)$ arose independently in the thesis of the author,
where it was related to the smallest symmetric group (or general
linear group) representation containing a fixed decomposable
tensor. Recall that $\S_n$ acts on the right of $V^{\otimes n}$ via
place permutation. If $u \in V^{\otimes n}$ is any tensor let $\S(u)$
be the smallest $\S_n$-representation in $V^{\otimes n}$ containing
$u$. If $f : E \to V$ is a realization of $M$, then it is easy to see
that the map
\[
(e_1 | p_1) (e_2 | p_2) \dots (e_n | p_n) \mapsto f(e_1) \otimes
f(e_2) \otimes \dots \otimes f(e_n) \in V^{\otimes n}
\]
extends to a unique surjective map of $\CC \S_n$-modules
\[
\CC\otimes U(M) \to \S( f(e_1) \otimes f(e_2) \otimes \dots \otimes
f(e_n) )
\]
The latter representation is a subtle projective invariant of the
vector configuration $f(E) \subset V$. For example, there is no known
example of two different realizations $f,g:E \to V$ of the same
matroid such that
\[
\S( f(e_1) \otimes f(e_2) \otimes \dots \otimes f(e_n) ) \not \approx
\S( g(e_1) \otimes g(e_2) \otimes \dots \otimes g(e_n) ),
\]
where $\approx$ is isomorphism of $\S_n$-modules.

\subsection{Which Irreducible Submodules Can Appear in $\CC \otimes U(M)$}
The irreducible representations of a the symmetric group $\S_n$ are
parametrized by partitions of $n$. Given a partition $\lambda$ and a
tableaux $T$ of shape $\lambda$ with content $(1,1,\dots,1)$ we can
construct the irreducible $\S_n$-representation indexed by $\lambda$ by
taking the left or right ideal in $\CC\S_n$ generated by the
\textit{Young symmetrizer}
\[
c_T = \left(\sum_{\sigma \in R_T} \operatorname{sign}(\sigma)
\sigma  \right) \left( \sum_{\tau \in C_T} \tau \right)
\]
where $R_T$ (respectively, $C_T$) is the subgroup of $\S_n$ preserving
the rows (respectively the columns) of $T$. We will say that the
partition $\lambda$ appears in a representation of $\S_n$ if it
contains a submodule isomorphic to the right ideal in $\CC\S_n$
generated by $c_T$. We say that a partition $\lambda$ has multiplicity
$m$ in $U(M)$ if the $c_T \CC\S_n$-isotypic component of $U(M)$ is
isomorphic to a direct sum of exactly $m$ copies of $c_T \CC\S_n$.

\begin{remark}
  Our indexing of the irreducible representations of $\S_n$ is the
  conjugate of the usual indexing. For example, $(n)$ corresponds to
  the sign representation and $(1^n)$ corresponds to the trivial
  representation.
\end{remark}
It follows from the our discussion above that if $f:E\to V$ is
realization of $M$ in a complex vector space $V$, and $\lambda$
appears in
\[
\S( f(e_1) \otimes f(e_2) \otimes \dots \otimes f(e_n) )
\]
then $\lambda$ appears in $\CC \otimes U(M)$ with positive
multiplicity.

The following result is equivalent to Gamas's Theorem on the vanishing
of symmetrized tensors (see \cite{berget}).
\begin{theorem}[Berget {\cite{berget}}]
  Let $f:E \to V$ be a realization of a matroid $M$ in a complex
  vector space $V$. A partition $\lambda$ appears in
  \[
  \S( f(e_1) \otimes f(e_2) \otimes \dots \otimes f(e_n) )
  \]
  if and only if there is a set partition of $E$ into independent sets
  whose sizes are the part sizes of $\lambda$.
\end{theorem}
\begin{corollary}\label{cor:gamas}
  Let $M$ be a matroid realizable over $\CC$. The partition $\lambda$
  appears in $U(M)$ if and only if $\lambda \leq \rho(M)$ in dominance
  order.
\end{corollary}
\begin{proof}
  One direction follows immediately from the previous theorem and Dias
  da Silva's Theorem~\ref{thm:dds}. It remains to prove the
  converse. The image of an antisymmetrizer
  $$\sum_{\sigma \in \S_k} \operatorname{sign}\sigma \in \CC\S_n,$$
  where $k \leq n$, on
  $(e_1 | p_1 )(e_2 | p_2) \dots (e_n | p_n)$
  is 
  \[
  (e_1e_2 \dots e_k | p_1 p_2 \dots p_k ) (e_{k+1}|p_{k+1} ) \dots
  (e_n | p_n).
  \]
  It follows that if some row of $T$ indexes a dependent set of $M$
  then $c_T$ applied to $(e_1 | p_1 )(e_2 | p_2) \dots (e_n | p_n)$ is
  zero. Since the projector of an $\S_n$-module to its $\lambda$-th
  isotypic component is, up to a scalar,
  \[
  \sum_{\sigma \in \S_n} \sigma c_T \sigma^{-1}
  \]
  we conclude that if every tableaux of shape $\lambda$ has a row
  indexing a dependent set of $M$ then $\lambda$ cannot appear in
  $\CC \otimes U(M)$.
\end{proof}

\subsection{Multiplicities of Hook Shapes}
A hook is a partition with at most one part not equal to one. Let
$\lambda^k$ denote the hook whose first part is $k$, and all other
parts are equal to one. In this subsection, we show how the
multiplicities of the irreducible $\S_n$ representations indexed by
hook shapes are related to the no broken circuit complex of $M$. The
results of this subsection hold for any matroid, regardless of
realizability.

To ease notation for the rest of this section, we assume that the
ground set of $M$ is $E = \{1,2,\dots,n\}$. We define a \textit{broken
  circuit} of $M$ as a circuit with its smallest element deleted. A
subset of the ground set of $M$ is said to be \textit{nbc} if it
contains no broken circuits of $M$. The collection of nbc sets of $M$
is a  simplicial complex called the \textit{nbc complex} of
$M$.

\begin{theorem}\label{thm:hook}
  The multiplicity $\lambda^k$ in $U(M)$ is the number of nbc sets of
  $M$ of size $k$ which contain the ground set element $1$.
\end{theorem}

\begin{example}
  Let $M$ be the matroid realizable over $\CC$ be the columns of the
  matrix
  \[
  \begin{bmatrix}
    1 & 0 & 0 & 1 & 1 & 0 \\
    0 & 1 & 0 & 1 & 0 & 1 \\
    0 & 0 & 1 & 0 & 1 & 1
  \end{bmatrix}
  \]
  Label the columns $1,2,\dots,6$, left to right.  The circuits of
  size three of $M$ are $124,136,235$, where $ijk$ denotes
  $\{i,j,k\}$, hence the broken circuits of $M$ are $24,36,35$. It
  follows that the no broken circuit sets of size $3$ are
  \[
  123,125,126,134,145,146,156
  \]
  and so the multiplicity of $\lambda^3$ in $\CC \otimes U(M)$ is
  $7$. For any ordering of the ground set of $M$, the smallest element
  is in at most two dependent sets of size $3$. It follows that for
  any ordering of the ground set, the number of standard Young
  tableaux of shape $\lambda^3$ whose first row indexes an independent
  set of $M$ is at least $8$.
\end{example}

\begin{remark}
  Even for matroids realizable over $\CC$ the standard basis theorem
  of Theorem~\ref{thm:whitney-freedom} does not hold. Indeed the
  previous example proves that the obvious spanning set of hook shaped
  tableaux $\tab(T_r|T_c) \in W(M)$ where $T_r$ has independent rows
  cannot be linearly independent.
\end{remark}

The nbc sets of $M$ with size $k$ that contain $1$ are precisely the
nbc bases of the truncation of $M$ to rank $k$. By the proof of
Lemma~\ref{lem:truncation} it is sufficient to prove the result when
$k$ is equal to the rank of $M$.

\begin{definition}
  If $D$ is a subset of $[n]$ we let $b_D \in \CC\S_n$ denote the
  antismmetrizer of the set $D$, i.e., $b_D = \sum_{\sigma \in \S_D}
  \operatorname{sign}(\sigma) \sigma$.

  For a given set $B \subset [n]$ we let $c_B$ denote the Young
  symmetrizer of the tableaux of shape $\lambda^{|B|}$ that has the
  elements of $B$ in its first row and the remaining elements $[n]-B$
  in the rows rows.
\end{definition}
It follows directly from the definition of the Young symmetrizer that
$b_B$ is a left factor of $c_B$.
\begin{proposition}
  Let $\<-\>$ denote taking the right ideal in $\CC\S_n$ generated by
  the elements $-$. There is an isomorphism of $\CC\S_n$-modules
  \[
  U(M) \approx \CC\S_n/ \< b_D : D \textup{ indexes a dependent set of
  }M\>
  \]
  induced by the map that sends $(e_1|p_1) (e_2 | p_2) \dots (e_n |
  p_n)$ to the image of $1$ in the quotient.
\end{proposition}
\begin{proof}
  The map that sends $1 \in \CC\S_n$ to $(e_1|p_1) (e_2 | p_2) \dots
  (e_n | p_n) \in \Super(L^-|P^+)$ is an isomorphism onto its
  image. One then verifies that, up to a sign, the image of the
  antisymmetrizer $b_D$ is the tableaux $\tab(T|S)$ of hook shape
  where the first row of $T$ and $S$ consists of the numbers in $D$.
\end{proof}
By taking the $\lambda^{r(M)}$-th isotypic component of
the quotient in the above proposition we immediately have the
following result.
\begin{corollary}
  Let $\<-\>$ denote taking the right ideal in $\CC\S_n$ generated by
  the elements $-$. There is an $\CC\S_n$-isomorphism between the
  $\lambda^{r(M)}$-th isotypic component of $\CC\otimes U(M)$ and the
  quotient
  \[
  \<c_B : B \text{ is any set of size }r(M) \>/ \<c_D : D \text{ is a
    dependent set of }M \>
  \]
\end{corollary}

\begin{lemma}
  Let $S$ be any set of $[n]$, from which we form the Young
  symmetrizer of the tableaux of shape $\lambda^{|S|}$ that has the
  elements of $S$ in its first row. In $\CC\S_n$ we have the equality
  \[
  c_S = \sum_{f \in S} c_{S \cup e -f}(ef)
  \]
  where $e$ is any element not in $S$.
\end{lemma}
\begin{proof}
  We have $c_{S \cup e -f}(ef) = (ef) c_S$ in $\CC\S_n$. It follows
  that we can write this statement as
  \[
  \left(1- \sum_{e \in S}(ef)\right)c_{S} =0
  \]
  We can write this as
  \[
  b_{S \cup e} c_S/|S|!
  \]
  where $b_{S \cup e}$ is the antisymmetrizer of the set $S \cup e$,
  since the antisymmetrizer is near idempotent. It is well known that
  every irreducible that appears in the left or right ideal generated
  by $b_{S \cup e}$ is larger than $\lambda^{|S|+1}$ in dominance
  order. Since $c_S$ generates an irreducible of shape $\lambda^{|S|}
  < \lambda^{|S|+1}$, the product of the two elements must be zero.
\end{proof}
\begin{corollary}
  The quotient
  \[
  \<c_B : B \text{ is any set of size }r(M) \>/ \<c_D : D \text{ is a
    dependent set of }M \>
  \]
  is generated by the Young symmetrizers of the nbc bases of $M$.
\end{corollary}
\begin{proof}
  Inducting on the number broken circuits contained in a given base,
  the proof follows at once from the lemma.
\end{proof}
To prove the what remains of the theorem, we will show that the ideal
generated by the Young symmetrizers of nbc bases does not meet the
ideal generated by the Young symmetrizers of dependent sets. That is,
we prove that
\[
\< c_B : B\textup{ is an nbc basis of } M\> \cap \< c_D : D \textup{
  is a dependent set of }M\> = 0.
\]
This will be done by straightening the latter Young symmetrizers into
sums of Young symmetrizers of standard Young tableaux, which in turn
is accomplished by a series of reductions. In the end, the proof comes
down to the well known fact that the right ideals generated by Young
symmetrizers of standard Young tableaux have null intersection.

For the rest of this section $D$ will denote a dependent set of $M$.
\begin{claim}
  We have
  \[
  \< c_D: D \textup{ contains two circuits}\> \subset \< c_D: 1 \in D\>.
  \]
\end{claim}
\begin{proof}
  If $1 \notin D$ and $D$ contains more than one circuit then each of
  the sets $D-e \cup 1$ is dependent since $D-e$ is. We have $c_D =
  \sum_{e \in D} c_{D-e \cup 1}(1e)$ which proves the result.
\end{proof}
The remainder of the proof is adapted from Las Vergnas and Forge
\cite{lasvergnas-forge}. We call a set unicyclic if it contains a
unique circuit. Using the circuit elimination axioms it can be shown
that $D$ is is unicyclic if and only if it contains an element $e$
such that $D-e$ is independent. The proof of the following claim is
exactly the same as the proof of the previous one.
\begin{claim}
  Let $cl(D)$ denote the closure of $D$ in $M$. We have the inclusion,
  \[
  \< c_D: D \textup{ unicyclic, }1 \in cl(D)\> \subset \< c_D: 1 \in D
  \textup{ dependent}\>.
  \]
\end{claim}
For the unicyclic sets where $1 \notin cl(D)$ note that we can write
$D = I \cup e$, where $I$ is independent and $e$ is the smallest
element of the unique circuit of $D$. For a general independent set
$I$, though, it is possible to choose many elements $e$ such that $e$
is the minimum element of a circuit of $I \cup e$. The external
activity of an independent set $I$, denoted $ex(I)$, is the number of
elements $e$ such that $I \cup e$ contains a unique circuit and $e$ is
the minimum element of that circuit. Let $Ex(I)$ denote the set of
elements $e$ such that $e$ is the minimum element of a circuit of $I
\cup e$.
\begin{claim}
  We have the inclusion,
  \begin{multline*}
    \<c_D: 1 \notin cl(D), D \textup{ unicyclic}\>\subset \< c_D: 1
    \in D \textup{ dependent}\>\\+ \< c_{I \cup Ex(I)}: ex(I)=1, 1
    \notin cl(I)\>
  \end{multline*}
\end{claim}
\begin{proof}
  Let $I \cup e$ be a unicyclic set that has $ex(I) > 1$ but $1 \notin
  Ex(I)$. Then there is an element $f\in Ex(I)-e$. We have
  \[
  c_{I \cup e} = \sum_{g \in I} c_{(I-g \cup f) \cup e}(fg)
  \]
  We see that for all $g \in I$, $(I-g \cup f) \cup e$ is unicyclic,
  does not contain $1$ in its closure and is lexicographically smaller
  than $I \cup e$. Assuming inductively that $c_{(I-g \cup f)\cup e}$
  is in the ideal
  \[
  \< c_D: 1 \in D \textup{ dependent}\>+ \< c_{I \cup Ex(I)}: ex(I)=1,
  1 \notin cl(I)\>
  \]
  we have that $c_{I \cup e}$ is in this ideal too.
\end{proof}
We now straighten the generators of the ideal
\[
\< c_{I \cup Ex(I)}: ex(I)=1, 1 \notin cl(I)\>.
\] 
If $I$ is independent of rank $r(M)-1$, has external activity equal to
one but does not contain $1$ in its closure then $I \cup 1$ is a
broken circuit base. However, for all elements $g \in I$, $(I-g \cup
1) \cup Ex(I)$ is a no broken circuit base of $M$. Since a Young
symmetrizer $c_S$, $|S|=r(M)$ is that of a standard Young tableau if
and only if $1 \in S$, we have proved
\begin{multline*}
  \<c_B: B \textup{ an nbc base} \>  \cap \< c_D: D \textup{ dependent}\> \\
  = \<c_B: B \textup{ an nbc base} \> \cap \< c_{I \cup ex(I)}:
  ex(I)=1, 1 \notin cl(I)\>.
\end{multline*}
Finally, every Young symmetrizer on the last ideal has support on a
unique broken circuit base containing $1$, so this intersection must
be zero. This completes the proof of the theorem.

\bibliography{whitney}\nocite{*}
\bibliographystyle{plain}
\end{document}